\documentclass[12pt]{amsart}
\usepackage{amssymb}
\usepackage{amsfonts}
\usepackage{amscd}
\usepackage[T1]{fontenc}

\swapnumbers

\newcommand{\Q}{\mathbb{Q}}
\newcommand{\F}{\mathbb{F}}
\newcommand{\Z}{\mathbb{Z}}
\newcommand{\N}{\mathbb{N}}

\newcommand{\Lmac}{\cL_{\rm Mac}}
\newcommand{\Lring}{\cL_{\rm ring}}

\def\11{{\mathbf 1}}
\def\AA{{\mathbf A}}

\def\FF{{\mathbb F}}

\def\PP{{\mathbf P}}
\def\QQ{{\mathbb Q}}

\def\ZZ{{\mathbb Z}}

\def\cL{{\mathcal L}}
\def\cM{{\mathcal M}}

\def\cO{{\mathcal O}}

\mathchardef\alphag="7C0B
\mathchardef\betag="7C0C
\mathchardef\gammag="7C0D
\mathchardef\deltag="7C0E
\mathchardef\varepsilong="7C22
\mathchardef\varphig="7C27
\mathchardef\psig="7C20
\mathchardef\zetag="7C10
\mathchardef\epsilong="7C0F
\mathchardef\rhog="7C1A
\mathchardef\taug="7C1C
\mathchardef\upsilong="7C1D
\mathchardef\iotag="7C13
\mathchardef\thetag="7C12
\mathchardef\pig="7C19
\mathchardef\sigmag="7C1B
\mathchardef\etag="7C11
\mathchardef\omegag="7C21
\mathchardef\kappag="7C14
\mathchardef\lambdag="7C15
\mathchardef\mug="7C16
\mathchardef\xig="7C18
\mathchardef\chig="7C1F
\mathchardef\nug="7C17
\mathchardef\varthetag="7C23
\mathchardef\varpig="7C24
\mathchardef\varrhog="7C25
\mathchardef\varsigmag="7C26
\mathchardef\Omegag="7C0A
\mathchardef\Thetag="7C02
\mathchardef\Sigmag="7C06
\mathchardef\Deltag="7C01
\mathchardef\Phig="7C08
\mathchardef\Gammag="7C00
\mathchardef\Psig="7C09
\mathchardef\Lambdag="7C03
\mathchardef\Xig="7C04
\mathchardef\Pig="7C05
\mathchardef\Upsilong="7C07

\newtheorem{thm}{Theorem}
\newtheorem{lem}{Lemma}
\newtheorem{cor}[thm]{Corollary}
\newtheorem{prop}[lem]{Proposition}
\newtheorem{claim}{Claim}

\newtheorem*{cor*}{Corollary}

\theoremstyle{remark}

\theoremstyle{plain}

\numberwithin{equation}{subsection}

\def\boxit#1#2{\setbox1=\hbox{\kern#1{#2}\kern#1}%
\dimen1=\ht1 \advance\dimen1 by #1
\dimen2=\dp1 \advance\dimen2 by #1
\setbox1=\hbox{\vrule height\dimen1 depth\dimen2\box1\vrule}%
\setbox1=\vbox{\hrule\box1\hrule}%
\advance\dimen1 by .4pt \ht1=\dimen1
\advance\dimen2 by .4pt \dp1=\dimen2 \box1\relax}

  {\par\medskip\noindent #1\par\begingroup%
    \advance\leftskip by 1em\advance\rightskip by 1em}%
  {\par\endgroup}

\newcommand{\ord}{\operatorname{ord}}

\begin{document}

\setcounter{tocdepth}{2} 

\title[Defining valuation rings in Henselian valued fields]{Uniformly defining  valuation rings in Henselian valued fields with finite or pseudo-finite residue fields} 


\author[R. Cluckers]{Raf Cluckers}
\address{Universit\'e Lille 1, Laboratoire Painlev\'e, CNRS - UMR 8524, Cit\'e Scientifique, 59655
Villeneuve d'Ascq Cedex, France, and,
Katholieke Universiteit Leuven, Department of Mathematics,
Celestijnenlaan 200B, B-3001 Leu\-ven, Bel\-gium\\}
\email{Raf.Cluckers@math.univ-lille1.fr}
\urladdr{http://math.univ-lille1.fr/$\sim$cluckers}

\author[J. Derakhshan]{Jamshid Derakhshan}
\address{University of Oxford, Mathematical Institute,
24-29 St Giles', Oxford OX1 3LB, UK}
\email{derakhsh@maths.ox.ac.uk}

\author[E. Leenknegt]{Eva Leenknegt}
\address{Purdue University, Department of Mathematics,
150 N. University Street,
West Lafayette, IN 47907-2067, US}
\email{eleenkne@math.purdue.edu}

%

\author[A. Macintyre]{Angus Macintyre}
\address{Queen Mary, University of London,
School of Mathematical Sciences, Queen Mary, University of London, Mile End Road, London E1 4NS, UK}
\email{angus@eecs.qmul.ac.uk}

\subjclass[2000]{Primary 11D88, 11U09; Secondary 11U05}
\keywords{Definability, Diophantine sets, Hilbert's Tenth Problem}

\begin{abstract}
We give a definition, in the ring language, of $\ZZ_p$ inside $\QQ_p$ and of $\FF_p[[t]]$ inside $\FF_p((t))$, which works 
uniformly for all $p$ and all finite field extensions of these fields, and in many other Henselian valued fields as well.
The formula can be taken existential-universal in the ring language, and in fact existential in a modification of the 
language of Macintyre. 
Furthermore, we show the negative result that in the language of rings there does not exist a uniform definition by an 
existential formula and neither by a universal formula for the valuation rings of all the finite extensions of 
a given Henselian valued field. We also show that there is no existential 
formula of the ring language defining 
$\Z_p$ inside $\Q_p$ uniformly for all $p$. For any fixed finite extension of $\Q_p$, we give an existential formula and a universal 
formula in the ring language which define the valuation ring.
\end{abstract}

\maketitle

\section{Introduction}

Uniform definitions of valuation rings inside families of 
Henselian valued fields have played important roles in the work related to 
Hilbert's 10th problem by B.~Poonen \cite{Poonen} and by J.~Koenigsmann \cite{Koeni}, especially uniformly in 
$p$-adic fields. We address this issue in a wider setting, using the ring language and Macintyre's language. 
Since the work \cite{Mac}, the Macintyre language has always been prominent in the study of $p$-adic fields.

Let $\cL_{\rm ring}$ be the ring language $(+,-,\cdot,0,1)$. 
Write $\Lmac$ for the language of Macintyre, which is obtained from $\cL_{\rm ring}$ by 
adding for each integer $n>0$ a predicate $P_n$ for the set of nonzero $n$-th powers. 
We assume that the reader is familiar with pseudo-finite fields and Henselian valued fields. 
For more information we refer to \cite{EngPres}, \cite{Marker}, \cite{CDM}, and \cite{NotesZoe}.

The following notational conventions are followed in this paper. 
For a Henselian valued field $K$ we will write $\cO_K$ for its valuation ring. $\cO_K$ is assumed nontrivial. 
$\cM_K$ is the maximal ideal of $\cO_K$, and $k=\cO_K/\cM_K$ is the residue field. We denote by $res$ the natural map 
$\cO_K\to k$, and by $\ord$ the valuation.

Given a ring $R$ and a formula $\varphi$ in $\cL_{\rm ring}$ or $\Lmac$ in $m\geq 0$ free variables, 
we write $\varphi(R)$ for the subset of $R^m$ consisting of the elements that satisfy $\varphi$. In this paper we 
will always work without parameters, that is, with $\emptyset$-definability.

\begin{thm}\label{main} 
There is an existential formula $\varphi(x)$ in $\cL_{\rm{ring}}\cup \{P_2,P_3\}$ such that 
$$
\cO_K = \varphi(K)
$$
holds for any Henselian
valued field $K$ with finite or pseudo-finite residue field $k$ 
provided that $k$ contains non-cubes in case its characteristic is 2.


\end{thm}
We are very grateful to an anonymous referee for pointing out to us that our argument in an earlier version 
failed when $k$ has characteristic $2$ and every element is a cube (i.e. $(k^*)^3=k^*$). There are such 
$k$, finite ones and pseudo-finite ones (cf. Section \ref{sec-app}).

Note that in such a case $k$ has no primitive cube root of unity, and so its unique quadratic extension is 
cyclotomic. That extension is the Artin-Schreier extension, and (as the referee suggested) it is appropriate to adjust the 
Macintyre language by replacing $P_2$ by $P_2^{AS}$, where 
$$
P_2^{AS}(x)\Leftrightarrow \exists y (x=y^2+y).
$$
This has notable advantages, namely:
\begin{thm}\label{main2} 
There is an existential formula $\varphi(x)$ in $\cL_{\rm{ring}}\cup \{P_2^{AS}\}$ such that 
$$
\cO_K = \varphi(K)
$$
holds for all Henselian
valued fields $K$ with finite or pseudo-finite residue field.
\end{thm}
Since in a field of characteristic not equal to $2$, we have $P_2^{AS}(x)\Leftrightarrow P_2(1+4x)$, 
Theorem \ref{main2} implies the following.
\begin{thm}\label{main'} 
There is an existential formula $\varphi(x)$ in $\cL_{\rm{ring}}\cup \{P_2\}$ such that 
$$
\cO_K = \varphi(K)
$$
holds for all Henselian
valued fields $K$ with finite or pseudo-finite residue field of characteristic not equal to $2$.
\end{thm}

Before proving the above theorems, we state some other results. First some negative results.
\begin{thm}\label{main3} Let $K$ be any Henselian valued field. There does not exist an existential formula
$\psi(x)$ in $\Lring$ such that
$$\mathcal{O}_L=\psi(L)$$
for all finite extensions $L$ of $K$. Neither does there exist a universal formula $\eta(x)$ in $\Lring$ such that
$$\mathcal{O}_L=\eta(L)$$
for all finite extensions $L$ of $K$.
\end{thm}

The following was noticed by the referee.
\begin{thm}\label{main4}
There is no existential or universal $\Lring$-formula $\varphi(x)$ such that $\Z_p = \varphi(\Q_p)$ for all the primes $p$. More 
generally, given any $N>0$, 
there is no such formula $\varphi(x)$ such that $\Z_p = \varphi(\Q_p)$ for all $p\geq N$.
\end{thm}

For a fixed local field of characteristic zero, we can give existential and universal definitions.
\begin{thm}\label{e-def}Let $K$ be a finite extension of $\QQ_p$. Then the valuation ring $\cO_K$ of $K$ is
definable by an existential formula in $\cL_{\rm ring}$ and also by a universal formula in $\cL_{\rm ring}$.
\end{thm}

\section{Negative results}
\subsection{Proof of Theorem \ref{main3}}
Suppose that there was such an existential formula $\psi(x)$. Let
$K^{alg}$ denote the algebraic closure of $K$. By \cite[Lemma 4.1.1 and Theorem 4.1.3]{EngPres}, there
is a unique valuation on $K^{alg}$ extending the valuation on $K$. The valuation ring $\cO_K$
has a unique prolongation to every algebraic extension of $K$. The valuation ring $\mathcal{O}_{K^{alg}}$
of $K^{alg}$ is the union of the valuation rings of the finite extensions $L$, and is thus
contained in $\psi(K^{alg})$. On the other hand, if $a\in \psi(K^{alg})$,
then $a\in \psi(L)$ for some
finite extension $L$ of $K$. Thus $a$ lies in the valuation ring of $L$, and hence
$a\in \mathcal{O}_{K^{alg}}$. So $\psi(K^{alg})$ coincides with the valuation ring of
$K^{alg}$ which implies that it must be finite or cofinite, contradiction.

We will now show that there is no existential formula $\theta(x)$ in the language of rings
such that for all finite extensions $L$ of $K$
$$\theta(L)=\mathcal{M}_L.$$
Suppose that there was such a formula $\theta(x)$. Then since the maximal ideal of $\mathcal{O}_{K^{alg}}$ is the union of
the maximal ideals $\mathcal{M}_L$ over all finite extensions $L$ of $K$, we see that
if $a\in \mathcal{M}_{K^{alg}}$, then $a\in \mathcal{M}_L$ for some finite extension $L$ of $K$, hence $\theta(a)$ holds in $L$,
so $\theta(a)$ holds in $K^{alg}$. Conversely, if $K^{alg}\models \theta(a)$, where $a\in K^{alg}$,
then $L\models \theta(a)$ for
some finite extension $L$ of $K$, hence $a\in \mathcal{M}_L$, thus $a\in \mathcal{M}_{K^{alg}}$. Therefore
$\theta(K^{alg})$ coincides with the maximal ideal of the valuation ring of $K^{alg}$ which implies that it must be finite
or cofinite, contradiction.

If $\theta(x)$ is a formula defining $\mathcal{M}_L$, then the formula
$$\sigma(x):=\exists z (xz=1 \wedge \theta(z))$$
defines the set $L \setminus \mathcal{O}_L$. We deduce that
there does not exist an existential formula $\sigma(x)$
in the language of rings such that for all finite extensions $L$ of $K$
$$\sigma(L)=L \setminus \mathcal{O}_L.$$
Thus there does not exist a universal formula $\eta(x)$
of the language of rings such that for all finite extensions $L$ of $K$
$$\eta(L)=\mathcal{O}_L.$$
The proof of Theorem \ref{main3} is complete.

\subsection{Proof of Theorem \ref{main4}}

Suppose there is such a formula $\varphi(x)$. By a result of Ax \cite[Proposition 7, pp.260]{Ax}, there is 
an ultrafilter $\mathcal{U}$ on the set $\PP$ of all primes such that the ultraproduct 
$k=(\prod_{p\in \PP} \F_p)/\mathcal{U}$ satisfies 
$$k\cap \Q^{alg}=\Q^{alg}.$$
The field $K=(\prod_{p\in \PP} \Q_p)/\mathcal{U}$ is Henselian with residue field $k$, which is pseudo-finite  
of characteristic zero, and value group an ultrapower of $\Z$. 


If $L$ is a finite extension of $K$, the residue field $k'$ of $L$ is a finite extension of $k$, hence is pseudo-finite and 
has the same algebraic numbers as $k$. Since two pseudo-finite fields with isomorphic subfields of algebraic numbers are 
elementarily equivalent (\cite[Theorem 4, pp.255]{Ax}), $k'\equiv k$. Thus 
all residue fields of finite extensions of $K$ are elementarily equivalent to $k$ 
and all value groups are elementarily equivalent to $\Z$. So, 
by the theorem of Ax-Kochen \cite[Theorem 3, pp.440]{AK}, $L\equiv K$ for all finite extensions $L$ of $K$, 
and so $\cO_L = \varphi(L)$ uniformly, contradicting Theorem \ref{main3}.

\section{Proof of Theorem \ref{e-def}}

Suppose $K$ has degree $n$ over $\QQ_p$. We have $n=ef$, where $f$ and $e$ are respectively the
residue field dimension and ramification index of $K$ over $\QQ_p$ (cf.~\cite{Frohlich}). 
Let $L$ be the maximal unramified extension of $\QQ_p$ inside $K$. $L$ has residue field $\FF_{p^f}$ and value group
$\Z$ for the valuation $\ord$ extending the $p$-adic valuation of $\QQ_p$. $K$ has value group $(1/e) \Z$ for the valuation
$\ord$. We denote by $| |$ the corresponding norm on $K$.

Select (non-uniquely) a monic irreducible polynomial $G_0(x)$ over $\FF_p$ of degree $f$ such that $\FF_{p^f}$ is the splitting
field of $G_0(x)$. Consider a monic polynomial $G(x)$ over $\Z$ which reduces to $G_0(x)$ mod $p$. The polynomial $G_0(x)$
has a simple root in $\FF_{p^f}$, so by Hensel's Lemma, $G(x)$ has a root $\gamma$ in $L$.
\begin{claim} $L=\QQ_p(\gamma)$.
\end{claim}
\begin{proof}[Proof of the claim] Clearly $\QQ_p(\gamma) \subset L$. But the residue field of $\QQ_p(\gamma)$ contains $\FF_{p^f}$. So the
dimension of $\QQ_p(\gamma)$ over $\QQ_p$ is at least $f$. So $L=\QQ_p(\gamma)$.\end{proof}

Note that $G(x)$ is irreducible over $\Z_p$ and so over $\QQ_p$, and $G(x)$ splits in $L$. 
Thus all the roots of $G(x)$ are conjugate over
$\QQ_p$ by automorphisms of $L$.
 We can choose an Eisenstein polynomial over $L$ of the form
$$x^e+H_{e-1}(\gamma)x^{e-1}+\dots+H_0(\gamma) \in L[x],$$
where for $i\in \{0,\dots,e-1\}$, $H_j(z)$ is a polynomial in the variable $z$ over $\QQ_p$. We aim to get an Eisenstein 
polynomial whose coefficients are in $\QQ(\gamma)$.
For any polynomials $H_0^*(z),\dots,H_{e-1}^*(z)$ over $\QQ$, we let 
$$H^*_z(x):=x^e+H^*_{e-1}(z)x^{e-1}+\dots+H^*_0(z) \in \Q(z)[x].$$
If $H_j^*(z)$ is such that $\lvert H_j(z)-H_j^*(z) \lvert$ is very small, then since $\ord(\gamma)\in \Z$, it follows that 
$\lvert H_j(\gamma)-H_j^*(\gamma)\lvert$ is also very small. Thus we can choose $H_j^*(z)$ over $\QQ$ 
sufficiently close to $H_j(z)$ 
so that $H^*_{\gamma}(x) \in \Q(\gamma)[x]$ is Eisenstein. So 
$H^*_{\gamma}(x)$ is irreducible over $L$, and, by Krasner's Lemma, it has a root in $K$ which generates
$K$ over $L$. For any other root $\gamma'$ of $G(x)$, there is a $\QQ_p$-automorphism $\sigma$ of $L$ such that
$\sigma(\gamma)=\gamma'$, and thus $\sigma(H^*_j(\gamma))=H^*_j(\gamma')$. 
Since $L$ is unramified over $\QQ_p$ and $p$ is a uniformizer in $L$, the
valuation ring of $L$ is definable without parameters and $\sigma$ preserves the valuation. Thus
$H^*_{\gamma'}(x)$ is also an Eisenstein polynomial. By \cite[Theorem 1, p.23]{Frohlich}, any root of an
Eisenstein polynomial is a uniformizer. We have thus shown that for any root $\eta$ of $G(x)$, any
root of $H^*_{\eta}(x)$ is a uniformizer. Indeed, $\{t: \exists \eta~G(\eta)=0 \wedge H^*_{\eta}(t)=0\}$ is 
an existentially definable nonempty set of uniformizers. So using Hensel's Lemma, we can define $\cO_K$ by
$$\exists z \exists y \exists w ~(G(z)=0 \wedge H^*_{z}(y)=0 \wedge 1+yx^2=w^2)$$
if $p\neq 2$, and 
$$\exists z \exists y \exists w ~(G(z)=0 \wedge H^*_{z}(y)=0 \wedge 1+yx^3=w^3)$$
if $p\neq 3$.

This completes the proof of existential definability of $\cO_K$. Note that combined with the remark about 
existential definition of a nonempty set of uniformizers, it gives existential definition of the set of 
uniformizers, and so of the maximal ideal $\cM_K$ as the set of elements of $K$ which are 
a product of a uniformizer and an element of $\cO_K$. Thus the complement of $\cO_K$ is existentially definable 
as the set of inverses of elements of $\cM_K$. Hence $\cO_K$ is universally definable.

\section{Proof of Theorems~\ref{main} and \ref{main2}}
For any prime number $p$, let $T_p(x)$ be the condition about $1$ free variable $x$ expressing that
$$ 
 p^p + x\in P_p \wedge x\not\in P_p.  
$$
Let $T(x)$ be the property about $x\in K$ saying that
$$ 
T_{2}(x) \vee T_{3}(x).
$$
Let $T^+(x)$ be the statement 
$$x\neq 0 \wedge \neg P_2^{AS}(x) \wedge \neg P_2^{AS}(x^{-1}).$$
Recall that $\wedge$ stands for conjunction and $\vee$ for disjunction in first-order languages.

\begin{lem}\label{dim}
Let $k$ be a pseudo-finite field. If the characteristic of $k$ is different from $2$, then $T_{2}(k)$ is infinite.
If the characteristic of $k$ is $2$ and $k$ contains a non-cube, then $T_{3}(k)$ is infinite.\end{lem}
\begin{proof} 
Suppose the characteristic of $k$ is different from $2$. $k$ is elementarily equivalent to an 
ultraproduct of finite fields $\F_q$ where $q$ is a power of an odd prime. Thus $(q-1,2)\neq 1$, hence 
$\F_q^\times$ contains a non-square (cf. Section \ref{sec-app}, Proposition \ref{app}). 
Thus $k^\times$ contains a non-square $a$. Then $T_2(x)$ is equivalent with
$$
\exists w,v  (w^2 = 4 + x \wedge a v^2 = x).
$$
Now consider the curve $C$ given by $w^2 = 4 + x,\ a v^2 = x$ in $\AA^3$. Since this is an absolutely irreducible curve 
defined over $k$, it follows by the pseudo-algebraic closedness of $k$ that $C(k)$ is infinite. Thus, $T_2(k)$ is infinite. 
The proof for characteristic $2$ is similar.\end{proof}

\begin{lem}\label{dim2} $T^+(k)$ is infinite for every pseudo-finite field $k$.
\end{lem}
\begin{proof}
Given a pseudo-finite field $k$ choose $a\in k \setminus P_2^{AS}(k)$ if $k$ has characteristic $2$ 
and $a\in k \setminus k^2$ if $k$ has characteristic different from $2$, and define the curve $\mathcal{C}_a$ by 
$$w^2+w=a-x$$
$$v^2+v=a-x^{-1}$$
if $k$ has characteristic $2$; and 
$$1+4x=aw^2$$
$$1+4x^{-1}=av^2$$
if $k$ has characteristic different from $2$. Then $\mathcal{C}_a$ is an absolutely irreducible curve in $\AA^3$. Since $k$ is 
pseudo-algebraically closed, $\mathcal{C}_a(k)$ is infinite. Note that 
$$T^+(x) \Leftrightarrow \exists v \exists w~(v,w,x)\in \mathcal{C}_a(k),$$
which completes the proof.
\end{proof}

\begin{lem}\label{val}
Let $K$ be any Henselian valued field with residue field $k$. 
Then, $T(K)$ is a subset of the valuation ring $\cO_K$ and $T^+(K)$ is a subset of the units $\cO_K^\times$. 
Moreover, $T(K)$ contains both the sets
$$
res^{-1}(T_{2}(k)\setminus\{0\})
\ \mbox{ and }\ res^{-1}(T_{3}(k)\setminus\{0\}),
$$
and $T^+(K)$ contains $res^{-1}(T^+(k))$.
\end{lem}
\begin{proof}
We first show that $T_2(K)\subset \cO_K$ for all Henselian valued fields $K$. 
It suffices to show for $x\in K\setminus \cO_K$ that $x$ is a square if and only if $x+4$ is a square. Let $x\in K\setminus 
\cO_K$. We show the left to right direction, the converse is similar. So assume $x$ is a square. 
It suffices to show that $1+4/x$ is a square, for then $x+4$ will be a product of two squares  
$1+4/x$ and $x$, hence a square.

Let $f(y):=y^2-1-4/x$. Since $|f'(1)|=|2|$ and $|x|>1$, we have 
$$|f(1)|=|4/x|<|4|=|2|^2=|f'(1)|^2.$$
Thus by Hensel's Lemma, $f(y)$ has a root in $\cO_K$. This shows that $T_2(K)\subset \cO_K$. 
One proceeds similarly to show that $T_3(K)\subset \cO_K$. It follows that $T(K)\subset \cO_K$ for all 
Henselian valued fields $K$. 

Now let $x\in T_2(k)\setminus \{0\}$. 
This implies that the characteristic of $k$ is not $2$. Thus if $\hat x\in \cO_K$ is any lift of $x$,
by Hensel's Lemma, $\hat x\in T_2(K)$, so $res^{-1}(T_2(k))\subset T_2(K)$. Similarly 
$x\in T_3(k)\setminus \{0\}$ implies that the characteristic of $k$ is not $3$, and $res^{-1}(T_3(k))\subset T_3(K)$. 
The other assertions concerning $T^+(K)$ and $T^+(k)$ are immediate. 
\end{proof}

We will use the following theorem of Chatzidakis - van den Dries - Macintyre \cite{CDM}. 
This result can be thought of as a definable version of the classical Cauchy - Davenport theorem.

\begin{thm}\cite[Proposition 2.12]{CDM}\label{CDM} Let $K$ be a pseudo-finite field and $S$ an infinite definable subset of $K$. 
Then every element of $K$ can be written as $a+b+cd$, with $a,b,c,d\in S$.
\end{thm}

\begin{cor}\label{cdm} Let $\varphi(x)$ be an $\cL_{\mathrm{ring}}$-formula 
such that $\varphi(k)$ is infinite for every pseudo-finite field $k$. Then there exists $N=N(\varphi)\geq 1$ such that 
$$K=\{a+b+cd: a,b,c,d \in \varphi(K)\}.$$
for every finite field $K$ of cardinality at least $N$.
\end{cor}

\begin{proof} Follows from Theorem \ref{CDM} and a compactness argument.\end{proof}

\begin{thm}\label{cor-to-cdm} 
Let $\varphi(x)$ be an $\cL_{\mathrm{ring}}$-formula such that $\varphi(k)$ is infinite for every pseudo-finite field $k$
and $\varphi(K)\subset \cO_K$ and $res^{-1}(\varphi(k))\subset \varphi(K)$ for every 
Henselian valued field $K$ with pseudo-finite residue field $k$. 
Then there exists $N\geq 1$ such that 
$$\cO_K=\{a+b+cd:~a,b,c,d\in \varphi(K)\}$$
for every Henselian valued field $K$ with finite or pseudo-finite residue field of cardinality at least $N$.
\end{thm}
\begin{proof} By assumption, the set on the right hand side is included in $\cO_K$, so we prove the other inclusion. 
Let $N$ be as in Corollary \ref{cdm}. Let $K$ denote a Henselian valued 
field with finite or pseudo-finite residue field of cardinality at least $N$. 
Let $\theta\in \cO_K$. Theorem \ref{CDM} and Corollary \ref{cdm} imply that 
$$res(\theta)=a+b+cd,$$
where $a,b,c,d\in \varphi(k)$. 
Let $\hat{b},\hat{c},\hat{d}$ denote elements of $\cO_K$ which map to $b,c,d$ respectively under the map $res$. So
$$res(\theta-(\hat{b}+\hat{c}\hat{d}))=a.$$
Thus 
$$\theta-(\hat{b}+\hat{c}\hat{d})\in \varphi(K).$$
Since $\hat{b},\hat{c},\hat{d}\in \varphi(K)$, the proof is complete.
\end{proof}

\begin{cor}\label{mainc} There exists $N>0$ such that
$$
\cO_K =\{a+b+cd:~a,b,c,d\in T(K)\}
$$
for any Henselian
valued field $K$ with finite or pseudo-finite residue field $k$ with cardinality at least $N$ provided that 
$k$ contains non-cubes in case its characteristic is $2$.
\end{cor}
\begin{proof} Immediate.\end{proof}
\begin{cor}\label{mainc2} There exists $N>0$ such that
$$
\cO_K =\{a+b+cd:~a,b,c,d\in T^+(K)\}
$$
for any Henselian
valued field $K$ with finite or pseudo-finite residue field $k$ with cardinality at least $N$.
\end{cor}
\begin{proof} Immediate.
\end{proof}

For any integer $\ell>0$, $K$ any field, and $X\subset K$ any set, let $S_\ell(X)$ be the set consisting of all $y\in K$ 
such that $y^{\ell} - 1 + x \in X$ for some $x\in X$.

\begin{prop}\label{finite}
Let $K$ be a Henselian valued field with finite residue field $k$ with $q_K$ elements. Let  $\ell$ be any positive integer 
multiple of $q_K(q_K-1)$. Then one has
$$
\cO_K =  \{0,1\}  +  S_{\ell}(T^+(K)),
$$
where the sumset of two subsets $A,B$ of $K$ consists of the elements $a+b$ with $a\in A$ and $b\in B$. 
If $k$ has a non-cube in case it has characteristic different from $3$, then one has
$$
\cO_K =  \{0,1\}  +  S_{\ell}(T(K)).
$$
\end{prop}
\begin{proof}
Since $\cO_K$ is integrally closed in $K$, for any $l>0$ and any Henselian valued field $K$, one has by Lemma \ref{val} that 
$$S_l(T(K))\subset \cO_K$$
and 
$$S_l(T^+(K))\subset \cO_K.$$

\begin{claim}\label{c-1} For any unit $y\in \cO_K$ there is a positive $\gamma$ in the value group such that 
$\ord(y^l-1)>\gamma$. \end{claim}
\begin{proof} There are two cases. Either the value group has a least positive element or it has 
arbitrarily small positive elements. Suppose the first case holds. 
Let $\pi$ denote an element of least positive valuation. 

We assume $K$ has residue field $\F_q$, with $q=p^f$. 
Fix a unit $y$. Let $a$ be a (not necessarily primitive) $(q-1)$-th root of unity such that 
$$|y-a|<1.$$
Note that $a$ exists by Hensel's Lemma since $y$ is a root of the polynomial $x^{q-1}-1$ modulo the maximal ideal 
and is clearly non-singular. 

Write $y$ as $a+b\pi$, where $b\in \cO_K$. Then 
$$y^l=1+ la^{l-1}b\pi+\dots+b^l\pi^l.$$
Note that the binomial coefficients are divisible by $l$, and hence by $q$ and thus by $\pi$ (as $\pi^e=p$ where 
$e$ is ramification index), and $l\geq 2$; therefore 
$$v(y^l-1)\geq 2.$$
This proves the Claim in the first case. In the second case, there are arbitrarily small positive elements in the value group 
and $y^l-1$ has some strictly positive valuation, hence $\gamma$ exists in this case.\end{proof}

\begin{claim}\label{c-2} Given $\gamma$ a positive element of the value group,
there is $a\in T(K)$ and $b\in T^+(K)$ such that $\ord(a)\leq \gamma$, 
$\ord(b)\leq \gamma$, and
$$a+a\cM_K\subset T(K)$$
$$b+b\cM_K\subset T^+(K).$$
\end{claim}
\begin{proof} Again, first assume that 
the value group has a least positive element $\pi$. Clearly $\pi$ is a non-square and a 
non-cube, and by Hensel's Lemma $4+\pi$ is a square if the residue characteristic is not equal to $2$, and 
$27+\pi$ is a cube if the residue characteristic is not equal to $3$. 
So we can take $a=\pi$, and by Hensel's Lemma we have 
$$a+a\cM_K\subset T(K).$$

In the case that there are elements of arbitrarily small positive value, there exist non-squares and non-cubes of 
arbitrarily small positive value. Indeed, fix a non-square $x$. We can choose $b$ such that its valuation is very close 
to half the valuation of $1/x$. Then $b^2x$ has valuation very close to zero. A similar argument works for the 
non-cubes. Then Hensel's Lemma as above completes the proof in this case. 

As for $T^+(K)$, given $\gamma>0$, choose any $b\in T^+(K)$. We have that $b$ is a unit and hence 
$\ord(b)=0<\gamma$. It follows from Hensel's Lemma that
$$b+b\cM_K\subset T^+(K)$$
since if $b+bm=y^2+y$ for some $y$, where $m\in \cM$, then 
$b-y^2-y$ has a non-singular root modulo the maximal ideal $\cM$; this 
contradicts $b\in T^+(K)$. This argument works for any value group.
\end{proof}
To complete the proof of the proposition take a unit $\alpha\in \cO_K$. 
By Claim \ref{c-1} there is $\gamma>0$ 
with
$$\ord(\alpha^l-1)>\gamma.$$
Choose elements $a\in T(K)$, 
and $b\in T^+(K)$ such that $\ord(a)\leq \gamma$ and $\ord(b)\leq \gamma$. Thus 
$$(\alpha^l-1)/a\in \cM_K$$
and 
$$(\alpha^l-1)/b\in \cM_K,$$
hence 
$$\alpha^l-1+a\in a+a\cM_K$$
and 
$$\alpha^l-1+b\in b+b\cM_K.$$
So by Claim \ref{c-2}, $\alpha\in S_l(T(K))$ and $\alpha\in S_l(T^+(K))$. 
This completes the proof.
\end{proof}


We can now give the proof of Theorems \ref{main} and \ref{main2}. 
By Lemma \ref{val}, for any $\ell>0$ and any Henselian valued field $K$ one has 
$$
S_{\ell}(T(K)) \subset  \cO_K.
$$
and 
$$
S_{\ell}(T^+(K)) \subset  \cO_K.
$$
From Proposition \ref{finite} and Corollary \ref{mainc} we deduce that 
there exists $\ell>0$ such that for any Henselian valued field $K$ with finite or pseudo-finite residue field we have  
\begin{equation}\label{cOK}
\cO_K = (\{0,1\} + S_{\ell}(T(K))) \cup \{a+b+cd:~a,b,c,d\in T(K)\}
\end{equation}
provided that the residue field $k$ contains a non-cube in case the characteristic of $k$ is $2$.
From Proposition \ref{finite} and Corollary \ref{mainc2} we deduce that 
\begin{equation}\label{cOK2}
\cO_K = (\{0,1\} + S_{\ell}(T^+(K))) \cup \{a+b+cd:~a,b,c,d\in T^+(K)\}
\end{equation}
for any Henselian valued field $K$ with finite or pseudo-finite residue field. 
Now Theorems \ref{main} and Theorem \ref{main2} follow since the unions in \ref{cOK} and 
\ref{cOK2} correspond to 
existential formulas in $\cL_{\rm{ring}}\cup \{P_2,P_3\}$ and $\cL_{\rm{ring}}\cup \{P_2^{AS}\}$ respectively.


\section{Appendix: Powers in pseudo-finite fields}\label{sec-app}
\begin{prop}\label{app} Let $p$ be a prime, $q$ a power of $p$, and $m\in \N$. The following are equivalent.
 \begin{itemize}
  \item $\F_q^*=(\F_q^*)^m$.
\item $(q-1,m)=1$.
\item $\F_h^*=(\F_h^*)^m$ for infinitely many powers $h$ of $p$.
 \end{itemize}
\end{prop}
\begin{proof} To show the first and second statements are equivalent, let $K=\F_q$. The multiplicative group 
$K^*$ is cyclic of order $q-1$. If 
$(m,q-1)=1$ then the map $x\rightarrow x^m$ is an automorphism of $K^*$. Conversely, 
if the map $x\rightarrow x^m$ from $K^*$ to $K^*$ is surjective, then it is injective. 
Choose $d$ with $d|m$ and $d|(q-1)$. There is $y$ such that $y^d=1$, so
$$y^m=(y^d)^{m/d}=1,$$
thus $y^m=1$, contradiction.

To prove the equivalence of the second and third statements, let $h$ be the order of $p$ in $(\Z/m\Z)^*$. 
Assume that $(p^s-1,m)=1$, for some $s$. For any $a\in \N$, we have
$$p^{ah+s}\equiv~p^{s}~(\mathrm{mod}~m),$$
hence 
$$p^{ah+s}-1\equiv~p^{s}-1~(\mathrm{mod}~m).$$
Therefore
$$(p^{ah+s}-1,m)=1.$$

Conversely, the last congruence shows that 
$(p^{ah+s}-1,m)=1$ implies $$(p^{s}-1,m)=1.$$
The proof is complete.\end{proof}
\begin{cor*} There are pseudo-finite fields of characteristic $2$ 
which do not contain non-cubes, and pseudo-finite fields of characteristic $3$ which do not contain non-squares. 
There are pseudo-finite fields $K$ of characteristic zero such that $K^*=(K^*)^n$ for all odd $n$.
\end{cor*}
\begin{proof} The first two statements are immediate by Proposition \ref{sec-app}. For the last statement use 
compactness to reduce to the case of finitely many $n$, therefore to one $n$ by taking product, and then use 
Proposition \ref{sec-app}.
\end{proof}
Note that the restriction to odd $n$ in the Corollary is necessary since for any finite field $k$ of odd characteristic, 
$k^*/(k^*)^2$ has cardinality $2$.

\subsection*{Acknowledgement}
The idea for this paper originated through discussions with J.~Koenigsmann, J~Demeyer, and C.~Degroote, to whom we are very grateful.
We are also indebted to E. Hrushovski and Z.~Chatzidakis for invaluable help on the results of \cite{Udifinite}. 
We also thank I. Halupczok and D. R. Heath-Brown for interesting discussions, and the referee for very valuable ideas.

\end{document}